\documentclass[twoside]{amsart}
\usepackage{amsfonts,amsthm}
\usepackage{enumerate}
\theoremstyle{plain}
\newtheorem*{thm A}{Theorem~A}
\newtheorem*{thm B}{Theorem~B}
\newtheorem*{thm C}{Theorem~C}
\newtheorem*{pro A}{Proposition~A}
\newtheorem*{pro B}{Proposition~B}
\newtheorem*{main1}{Theorem~1}
\newtheorem*{main2}{Theorem~2}
\newtheorem*{main3}{Theorem~3}

\newtheorem{theorem}{Theorem}[section]
\newtheorem{corollary}[theorem]{Corollary}

\newtheorem{remark}[theorem]{Remark}

\newtheorem{lemma}[theorem]{Lemma}

\def \N{\nabla}
\def \GBt{G_2({\mathbb C}^{m+2})}
\def \GBo{G_2({\mathbb C}^{m+1})}
\def \EN{{\eta}_{\nu}}
\def \EoK{{\eta}_1({\xi})}
\def \EKo{{\eta}({\xi}_1)}
\def \ENK{{\eta}_{\nu}({\xi})}
\def \al{\alpha}
\def \x{\xi}
\def \p{\phi}

\def \X{X_{0}}
\def \e{\eta}
\def \xo{\xi_{1}}

\def \etw{\eta_{2}}
\def \eth{\eta_{3}}
\def \xtw{{\xi}_2}
\def \xth{{\xi}_3}
\def \po{\phi_{1}}
\def \xt{\xi_{2}}
\def \xh{\xi_{3}}
\def \E{\eta}

\def \Dc{{\mathfrak D}^{\bot}}

\def \QP{{\mathcal Q}^{\bot}}
\def \Q{\mathcal Q}
\def \SN{{\sum}_{{\nu}=1}^3}
\def \PN{{\phi}_{\nu}}
\def \Rx{R_{\xi}}
\def \RN{\Bar{R}_N}

\def \sms{semi-parallel shape operator}
\def \smsj{semi-parallel structure Jacobi operator}

\def \BE{\begin{equation}}
\def \EE{\end{equation}}
\def \BEN{\begin{equation*}}
\def \EEN{\end{equation*}}
\def \BSN{\begin{split}}
\def \ESN{\end{split}}

\begin{document}
\title[Semi-parallel symmetric operators]{Semi-parallel symmetric operators for Hopf hypersurfaces in complex two-plane Grassmannians}

\vspace{0.2in}

\author[D.H. Hwang, H. Lee, and C. Woo]{Doo Hyun Hwang, Hyunjin Lee, and Changhwa Woo}
\address{\newline
D.H. Hwang, H. Lee and C. Woo
\newline Department of Mathematics,
\newline Kyungpook National University,
\newline Daegu 702-701, REPUBLIC OF KOREA}
\email{engus0322@knu.ac.kr}
\email{lhjibis@hanmail.net}
\email{legalgwch@knu.ac.kr}

\thanks{* This work was supported by Grant Proj. No. NRF-2011-220-C00002 from National Research Foundation of Korea.
The second author by Grant Proj. No. NRF-2012-R1A1A3002031 and the third supported by NRF Grant funded by the Korean Government (NRF-2013-Fostering Core Leaders of Future Basic Science Program).}

\begin{abstract}
In this paper, we introduce new notions of semi-parallel shape operators and structure Jacobi operators in complex two-plane Grassmannians $G_2({\mathbb C}^{m+2})$. By using such a semi-parallel condition, we give a complete classification of Hopf hypersurfaces in $\GBt$.
\end{abstract}

\maketitle

\section*{Introduction}\label{section 0}
\setcounter{equation}{0}
\renewcommand{\theequation}{0.\arabic{equation}}
\vspace{0.13in}


The classification of real hypersurfaces in Hermitian symmetric space is one of interesting parts in the field of differential geometry. Among them, we introduce a complex two-plane Grassmannian~$\GBt$ defined by the set of all complex two-dimensional linear subspaces in ${\mathbb C}^{m+2}.$ It is a kind of Hermitian symmetric space of compact irreducible type with rank~$2$. Remarkably, the manifolds are equipped with both a K\"{a}hler structure $J$ and a quaternionic K\"{a}hler structure ${\mathfrak J}$ satisfying $JJ_{\nu}=J_{\nu}J$ $(\nu=1,2,3)$ where $J_{\nu}$ is an orthonormal basis of $\mathfrak J$. When $m=1$, $G_2({\Bbb C}^3)$ is isometric to the two-dimensional
complex projective space ${\mathbb C}P^2$ with constant holomorphic
sectional curvature eight. When $m=2$, we note that the isomorphism $\text{Spin}(6) \simeq \text{SU}(4)$ yields an isometry between $G_2({\Bbb C}^4)$ and the real Grassmann
manifold $G_2^+({\mathbb R}^6)$ of oriented two-dimensional linear
subspaces in ${\Bbb R}^6$. In this paper, we assume $m \geq 3$. (see Berndt and Suh~\cite{BS1} and~\cite{BS2}).

\vskip 3pt

Let $M$ be a real hypersurface in $\GBt$ and $N$ a local unit normal vector field of $M$. Since $\GBt$ has the K\"{a}hler structure $J$, we may define a {\it Reeb vector field} $\xi$ defined by $\xi =-JN$ and a 1-dimensional distribution $[\xi]=\text{Span}\{\,\xi\}$. The Reeb vector field~$\xi$ is said to be a {\it Hopf} if it is invariant under the shape operator $A$ of $M$. The 1-dimensional foliation of $M$ by the integral curves of $\xi$ is said to be a {\it Hopf foliation} of $M$. We say that $M$ is a {\it Hopf hypersurface} if and if the Hopf foliation of $M$ is totally geodesic. By the formulas in \cite[Section~$2$]{LS}, it can be easily checked that $\xi$ is Hopf if and only if $M$ is Hopf.

\vskip 3pt

From the quaternionic K\"{a}hler structure $\mathfrak J$ of $\GBt$, there naturally exists {\it almost contact 3-structure} vector field $\xi_{1},\xi_{2},\xi_{3}$ defined by $\xi_{\nu}=-J_{\nu}N$, $\nu=1,2,3$. Put $\QP = \text{Span}\{\,\xi_1, \xi_2, \xi_3\}$, which is a 3-dimensional distribution in a tangent vector space $T_{x}M$ of $M$ at $x \in M$. In addition, $\Q$ stands for the orthogonal complement of $\QP$ in $T_{x}M$. It becomes the quaternionic maximal subbundle of $T_{x}M$. Thus the tangent space of $M$ consists of the direct sum of $\Q$ and $\QP$ as follows: $T_{x}M =\Q\oplus \QP$.

\vskip 3pt

For two distributions $[\xi]$ and $\QP$ defined above, we may consider two natural invariant geometric properties under the shape operator $A$ of $M$, that is, $A [\xi] \subset [\xi]$ and $A\QP \subset \QP$. By using the result of Alekseevskii \cite{Al-01}, Berndt and Suh~\cite{BS1} have classified all real hypersurfaces with two natural invariant properties in $\GBt$ as follows:
\begin{thm A}\label{theorem A}
Let $M$ be a connected real hypersurface in $\GBt$,
$m \geq 3$. Then both $[\xi]$ and $\QP$ are invariant
under the shape operator of $M$ if and only if
 \begin{enumerate}[\rm(A)]
\item {$M$ is an open part of a tube around a totally geodesic
$\GBo$ in $\GBt$, or} \item {$m$ is
even, say $m = 2n$, and $M$ is an open part of a tube around a
totally geodesic ${\mathbb H}P^n$ in $\GBt$}.
\end{enumerate}
 \end{thm A}

\noindent In the case~(A), we call $M$ is a real hypersurface of Type~$(A)$ in $\GBt$. Similarly in the case~(B) we call $M$ one of Type~$(B)$. Using Theorem~$\rm A$, many geometricians have given some characterizations for Hopf hypersurfaces in $\GBt$ with geometric quantities, for example, shape operator, normal (or structure) Jacobi operator, Ricci tensor, and so on. In particular, Lee and Suh~\cite{LS} gave a characterization for real hypersurfaces of Type~$(B)$ as follows:
\begin{thm B}\label{theorem B}
Let $M$ be a connected orientable Hopf hypersurface in $\GBt$, $m
\geq 3$. Then the Reeb vector field $\xi$ belongs to the
distribution $\Q$ if and only if $M$ is locally congruent
to an open part of a tube around a totally geodesic ${\mathbb H}P^n$
in $\GBt$, $m=2n$, where the distribution
$\Q$ denotes the orthogonal complement of $\QP$ in $T_{x}M$, $x \in M$. In other words, $M$ is locally congruent to real hypersurfaces of Type~$(B)$.
\end{thm B}
\vskip 3pt

On the other hand, regarding the parallelism of tensor field $T$ of type $(1,1)$, that is, $\N T =0$, on $M$ in $\GBt$, $m \geq 3$, there are many well-known results. Among them, when $T=A$ where $A$ denotes the shape operator of $M$, some geometricians have verified non-existence properties and some characterizations for the shape operator $A$ with many kinds of parallelisms, such as Levi-civita parallel, $ \mathfrak F$-parallel, $\QP$-parallel, Reeb parallel or generalized Tanaka-Webster parallel, and so on (see ~\cite{JLS}, \cite{LCW}, \cite{Suh01}, \cite{Suh02}, etc.).

\vskip 3pt

Furthermore, many geometricians considered such a parallelism for another tensor field of type (1,1) on $M$, namely, the Jacobi operator $R_{X}$ defined $(R_{X}(Y))(p)= (R(Y, X)X)(p)$, where $R$ denotes a Riemannian curvature tensor of type (1,3) on~$M$ and $X$, $Y$ denote tangent vector fields on $M$. Clearly, each tangent vector field~$X$ to $M$ provides the Jacobi operator $R_{X}$ with respect to $X$. When it comes to $X=\xi$, the Jacobi operator $R_\xi$ is said to be a {\it structure Jacobi operator}. Related to the tensor field $R_{\x}$ of type (1,1) on $M$, P\'erez, Jeong, and Suh~\cite{JPS} considered the parallelism, that is, ${\nabla}_{X}R_{\xi}=0$ for any $X \in TM$ and obtained a non-existence property.

\vskip 3pt

In this paper we consider a generalized notion for parallelism of tensor field of type (1,1) on $M$ in $\GBt$, namely, semi-parallelism. Actually, in \cite{ChoKi} a tensor field $F$ of type $(1,s)$ on a Riemannian manifold is said to be {\it semi parallel} if $R \cdot F=0$. It means that the Riemannian curvature tensor $R$ of~$M$ acts as a derivation on $F$. From this, it is natural that if a tensor field $T$ of type (1,1) is parallel, then $T$ is said to be a {\it semi-parallel}. Geometricians have proved various results concerning the semi-parallelism conditions of real hypersurfaces in complex space form (see \cite{ChoKi}, \cite{NieRyan}, \cite{PeSan}). Recently, K. Panagiotidou and M.M. Tripathi suggested the notion of {\it semi-parallel normal Jacobi operator} for a real hypersurface in $\GBt$ (see~\cite{PT}).

\vskip 3pt

Motivated by these works, we consider semi-parallelisms of the shape operator and the structure Jacobi operator for real hypersurfaces in $\GBt$, and assert the following theorems, respectively:
\begin{main1}
Let $M$ be a connected real hypersurface in complex two-plane Grassmannians $\GBt$, $m \geq 3$.
There does not exist  Hopf hypersurfaces $M$ with \sms \, if the smooth function $\al=g(A\x, \x)$ is constant along the direction of $\xi$.
\end{main1}
\vskip 3pt
\begin{main2}
Let $M$ be a connected real hypersurface in complex two-plane Grassmannians $\GBt$, $m \geq 3$.
There does not exist  Hopf hypersurfaces $M$ with \smsj \, if the smooth function $\al=g(A\x, \x)$ is constant along the direction of $\xi$.
\end{main2}

\vskip 3pt

In \cite{PT}, K. Panagiotidou and M.M. Tripathi proved the following
\begin{thm C}
There does not exist any connected Hopf hypersurface in complex two-plane Grassmannians $\GBt$, $m \geq 3$, with semi-parallel normal Jacobi operator if the smooth function $\al=g(A\x, \x)\neq 0$ and $\Q$- or $\QP$-component of $\x$ is invariant by the shape operator.
\end{thm C}

\noindent From this we consider that $M$ has a vanishing geodesic Reeb flow when it comes to normal Jacobi operator. Hence by virtue of \cite[Lemma 3.1]{LKS02}, it gives us a extended result with respect to Theorem~$\rm C$ as follows.
\begin{main3}\label{main3}
Let $M$ be a connected real hypersurface in complex two-plane Grassmannians $\GBt$, $m \geq 3$. There does not exist
Hopf hypersurfaces $M$ with normal Jacobi operator if the smooth function $\al=g(A\x, \x)$ is constant along the direction of $\xi$.
\end{main3}

\vskip 3pt
In this paper, we refer \cite{Al-01}, \cite{BS1}, \cite{BS2}, \cite{LS} and \cite{JMPS}, \cite{Suh01}, \cite{Suh02} for Riemannian geometric structures of $\GBt$ and its geometric quantities, respectively.

\vspace{0.15in}

\section{Semi-parallel shape operator}\label{section 1}

\setcounter{equation}{0}
\renewcommand{\theequation}{1.\arabic{equation}}
\vspace{0.13in}
In this section, let $M$ represent a Hopf real hypersurface in $\GBt$, $m \geq 3$, and~$R$ denote the Riemannian curvature tensor of $M$. Hereafter unless otherwise stated, we consider that $X,Y$, and $Z$ are any tangent vector field on $M$. Let $W$ be any tangent vector field on $\Q$.

\vskip 3pt

We first give the fundamental equation for the semi-parallelism of a tensor field $T$ of type (1,1) on $M$ and prove our Theorem~$1$.

\vskip 3pt

As mentioned in the introduction, a tensor field $T$ on $M$ is said to be semi parallel, if $T$ satisfies $R \cdot T=0$. It is equal to
\begin{equation*}\label{mcon}
(R(X,Y)T)Z=0.
\tag {\dag}
\end{equation*}
Since $(R(X,Y)T)Z=R(X,Y)(TZ)-T(R(X,Y)Z)$, the equation~\eqref{mcon} is equivalent to the following
\BE\label{mcon 02}
R(X,Y)(TZ)=T(R(X,Y)Z).
\tag{\ddag}
\EE

\vskip 5pt

Using this discussion, let us prove our Theorem~$1$ given in Introduction. In order to do this, suppose that $M$ has the semi-parallel shape operator, that is, the shape operator~$A$ of $M$ satisfies the condition $(R(X,Y)A)Z=0$. From the relation between \eqref{mcon} and \eqref{mcon 02}, we see that the given condition is equivalent to
\begin{equation}\label{eq: 3.2}
R(X,Y)(AZ)=A(R(X,Y)Z).
\end{equation}
Therefore from \cite[The equation of Gauss]{Suh01}, it becomes
\begin{equation}\label{eq: 3.3}
\begin{split}
& \quad \ g(Y,AZ)X - g(X,AZ)Y + g(\phi Y,AZ)\phi X - g(\phi X, AZ)\phi Y  \\
& \quad \ -2g(\phi X, Y)\phi AZ +g(AY,AZ)AX -g(AX, AZ)AY \\
& \quad \ + \ \sum_{\nu} \Big \{g(\phi_{\nu}Y, AZ) \phi_{\nu}X - g(\phi_{\nu}X, AZ) \phi_{\nu}Y -2g(\phi_{\nu}X, Y) \phi_{\nu} AZ \Big\}\\
& \quad \ + \  \sum_{\nu}\Big \{g(\phi_{\nu}\phi Y, AZ) \phi_{\nu}\phi X - g(\phi_{\nu}\phi X, AZ) \phi_{\nu}\phi Y \Big\}\\
& \quad \ - \  \sum_{\nu}\Big\{\eta(Y) \eta_{\nu}(AZ) \phi_{\nu}\phi X -\eta(X) \eta_{\nu}(AZ) \phi_{\nu}\phi Y \Big\}\\
& \quad \ -\sum_{\nu} \Big\{\eta(X) g(\phi_{\nu}\phi Y, AZ) - \eta(Y) g(\phi_{\nu}\phi X, AZ)\Big\}\xi_{\nu}\\
& = g(Y,Z)AX - g(X,Z)AY + g(\phi Y,Z)A\phi X - g(\phi X, Z)A\phi Y  \\
& \quad \ -2g(\phi X, Y)A\phi Z +g(AY,Z)A^{2}X -g(AX, Z)A^{2}Y \\
& \quad \  + \ \sum_{\nu} \Big \{g(\phi_{\nu}Y, Z) A\phi_{\nu}X - g(\phi_{\nu}X, Z) A\phi_{\nu}Y -2g(\phi_{\nu}X, Y) A\phi_{\nu} Z \Big\}\\
& \quad \  + \  \sum_{\nu}\Big \{g(\phi_{\nu}\phi Y, Z) A\phi_{\nu}\phi X - g(\phi_{\nu}\phi X, Z) A\phi_{\nu}\phi Y \Big\}\\
& \quad \  - \  \sum_{\nu}\Big\{\eta(Y) \eta_{\nu}(Z) A\phi_{\nu}\phi X -\eta(X) \eta_{\nu}(Z) A\phi_{\nu}\phi Y \Big\}\\
& \quad \ -\sum_{\nu} \Big\{\eta(X) g(\phi_{\nu}\phi Y, Z) - \eta(Y) g(\phi_{\nu}\phi X, Z)\Big\}A\xi_{\nu},\\
\end{split}
\end{equation}
where $\sum_{\nu}$ moves from $\nu=1$ to $\nu=3$.

\def \sn{\sum_{\nu}}

Putting $Y=Z=\x$ and using the condition of Hopf, the equation~\eqref{eq: 3.3} can be reduced to
\begin{equation}\label{eq: 3.4}
\begin{split}
& \ \ AX +\al A^{2}X  \\
& \quad \ - \sn \Big\{ \big(\EN (X) - \eta (X) \EN (\xi)\big) A\xi_{\nu} + 3\EN
(\phi X) A\PN \xi + \EN(\xi) A\PN \phi X \Big\}\\
& = \al X +\al^{2} AX  \\
\end{split}
\end{equation}
\begin{equation*}
\begin{split}
&\quad \ - \al\sn \Big\{ \big(\EN (X) - \eta (X) \EN (\xi)\big) \xi_{\nu} + 3\EN
(\phi X) \PN \xi + \EN(\xi) \PN \phi X \Big\}.
\end{split}
\end{equation*}

Our first purpose is to show that  $\x$ belongs to either $\Q$ or $\QP$.
\begin{lemma}\label{lemma 3.1}
Let $M$ be a Hopf hypersurface with \sms \ in $\GBt$, $m \geq 3$. If the principal curvature $\alpha=g(A\x, \x)$ is constant along the direction of Reeb vector field~$\xi$,
then $\xi$ belongs to either the distribution $\Q$ or the distribution $\QP$.
\end{lemma}

\begin{proof}
\noindent We consider that $\xi$ satisfies
\begin{equation*}\label{xi}
\xi = \eta(X_{0})X_{0}+\eta(\xi_{1})\xi_{1}
\tag{*}
\end{equation*}
for some unit vectors $X_{0} \in \Q$, $\xi_{1} \in \QP$, and $\eta(X_{0})\eta(\xi_{1})\neq 0$.

By virtue of \cite[Equation~$(2.10)$]{JMPS} and the assumption of $\xi \alpha =0$, we get $A \X = \alpha \X$ and $A \xo = \alpha \xo$.

\vskip 3pt

In the case of $\alpha=0$, using the equation in \cite[Lemma~$1$]{BS1},
\begin{equation}\label{eq: 1.4}
Y{\alpha}=({\xi}{\alpha}){\eta}(Y)-4{\SN}{\ENK}{\EN}({\phi}Y),
\end{equation}
we obtain that $\xi$ belongs to either $\Q$ or $\QP$. We next consider the case~$\alpha \neq 0$.

\vskip 3pt
We next consider the case~$\alpha \neq 0$.

\noindent Substituting $X=\p\X$ in \eqref{eq: 3.4} and using basic formulas including \eqref{xi}, we get
\BE
\begin{split}\label{eq: 3.5}
&A\p\X-3\e(\X)\EoK A\po\x+\EoK A\po\X-\EoK\e(\X)A\po\x+\al A^{2}\p\X \\
&=\al\p\X-3\al\e(\X)\EoK \po\x+\al\EoK \po\X-\al\EoK\e(\X)\po\x+\al^{2} A\p\X.
\end{split}
\EE

\noindent From \eqref{xi} and $\p \x =0$, we obtain that $\po\x=\e(\X)\po\X$ and $\p\X=-\e(\xo)\po\X$. In addition, substituting $X$ by $X_0$ into \cite[Lemma~$2.2$]{JMPS} and applying $A\X =\alpha \X$, we see that both vector fields $\p\X$ and $\po \X$ are principal with same corresponding principal curvature $k=\frac{\al^{2}+4\e^{2}(\X)}{\al}$.
 From this, \eqref{eq: 3.5} gives
\BEN
-4k\e^{2}(\X)\p\X+\al k^{2}\p\X-4\al\e^{2}(\X)\p\X-\al^{2} k\p\X=0.
\EEN
Since $\al \neq 0$, multiplying $\al$ to this equation, we obtain
\BEN
4\e^{2}(\X)(8\e^{2}(\X)+\al^{2})\p\X=0.
\EEN
By our assumptions, we get $\e(\X)\e(\xo)\neq 0$ which means ${\phi}X_0=0$. This makes a contradiction. Accordingly, we get a
complete proof of our Lemma.
\end{proof}

From Lemma~\ref{lemma 3.1}, we only have two cases, $\xi\in\Q$ or $\xi\in\QP$, under our assumptions. Next we further study the case $\x \in \QP$.

\begin{lemma}\label{lemma 3.2}
Let $M$ be a Hopf hypersurface with \sms \ in $\GBt$, $m \geq 3$. If the Reeb vector field $\xi$ belongs to the distribution $\QP$, then~$M$ must be a $\QP$-invariant hypersurface.
\end{lemma}

\begin{proof}
Since $\x \in \QP$, we may put $\xi=\xi_{1}\in \QP$ for the sake of our convenience.
Differentiating $\xi=\xi_1$ along any direction $X \in TM$ and using fundamental formulae in \cite[Section~$2$]{LS},
 it gives us
\begin{equation}\label{eq: 2.7}
\phi AX=2\eta_{3}(AX)\xi_{2}-2\eta_{2}(AX)\xi_{3}+\phi_{1} AX.
\end{equation}
Taking the inner product of \eqref{eq: 2.7} with $W \in \Q$ and taking symmetric part, we also have
\BE\label{eq: 3.16}
A\p W=A\po W.
\EE

Putting $X=\xtw$ and $X=\xth$  into \eqref{eq: 3.4}, we get, respectively,
\begin{equation*}\label{eq: 3.9}
\left \{
\begin{aligned}
2A\xtw+\al A^{2}\xtw & = 2\al\xtw+\al^{2}A\xtw, \\
2A\xth+\al A^{2}\xth & = 2\al\xth+\al^{2}A\xth.
\end{aligned}
\right.
\end{equation*}

For $\al=0$, clearly $\QP$ is invariant under the shape operator, i.e., $A\QP \subset \QP$.
Thus, let us consider $\al \neq 0$. Then the previous equations imply that
\begin{equation}\label{eq: 3.11}
\left \{
\begin{aligned}
A^{2}\xtw & = \frac{\al^{2}-2}{\al}A\xtw+2\xtw, \\
A^{2}\xth & = \frac{\al^{2}-2}{\al}A\xth+2\xth.
\end{aligned}
\right.
\end{equation}
Moreover, restricting $X=\xtw$, $Y=\xth$ and putting $Z=W \in \Q$, the equation~\eqref{eq: 3.3} becomes
\begin{equation*}\label{eq: 3.13}
\begin{split}
& 4\eth(AW)\xtw-4\etw(AW)\xth+2\p AW-2\po AW +\eth(A^{2}W)A\xtw-\etw(A^{2}W)A\xth \\
& \quad = 2A\p W-2A\po W+\eth(AW)A^{2}\xtw-\etw(AW)A^{2}\xth.
\end{split}
\end{equation*}
Applying \eqref{eq: 2.7}, \eqref{eq: 3.16} and \eqref{eq: 3.11} to this equation, it follows $ \eth(AW)\xtw=\etw(AW)\xth$. This means $\eth(AW)=\etw(AW)=0$ for any tangent $W\in\Q$. It completes the proof.
\end{proof}

From this lemma, we see that $M$ satisfying the assumptions in Lemma~\ref{lemma 3.2} is locally congruent to a model space of Type~$(A)$ in $\GBt$. Now, if we assume $\x \in \Q$, then $M$
with semi-parallel shape operator is locally congruent to one of Type~$(B)$ by virtue of Theorem~$\rm B$.

\vskip 3pt

Summing up these discussions, we conclude: \emph{let $M$ be a Hopf hypersurface in $\GBt$, $m \geq 3$. If $M$ satisfies \eqref{eq: 3.2} and $\x \al = 0$, then $M$ must be a model space of Type~$(A)$ or $(B)$.}

\vskip 3pt

Hereafter, let us check whether the shape operator of a model space of Type~$(A)$ (or one of Type~$(B)$) satisfies the semi-parallel condition~\eqref{eq: 3.2} by
\cite[Proposition~$\rm 3$]{BS1} (or \cite[Proposition~$\rm 2$]{BS1}, respectively).

\vskip 3pt

Let $M_{A}$ be a model space of Type~$(A)$ in $\GBt$. To show our purpose, we suppose that $M_{A}$ has the semi-parallel shape operator. From ~\eqref{eq: 3.4}, \cite[Proposition~$\rm 3$]{BS1}, and $\xi \in \QP$, we have
\begin{equation*}
(\lambda-\al)(2+\al\lambda)X=0
\end{equation*}
for any tangent vector $X \in T_{\lambda}=\{X \in T_{x}M|\, X \bot \x_{\nu}, \ \p X = \po X, \ x \in M \}$. Since
$\al = \sqrt{8} \cot\sqrt{8}r$ and $\lambda= -\sqrt{2}\tan\sqrt{2}r$ where $r \in (0, \pi / \sqrt{8})$,
it implies that every $X \in T_{\lambda}$ is a zero vector. This gives rise to a contradiction.
In fact, the dimension of the eigenspace $T_{\lambda}$ is $2m-2$ where $m \geq 3$.

\vskip 3pt

Now let us consider our problem for a model space of Type~$(B)$ denoted by~$M_{B}$. Similarly, we assume that the shape operator of $M_{B}$ is semi-parallel.
By virtue of \cite[Proposition~$\rm 2$]{BS1}, we see that $\x$ of $M_{B}$ belongs to $\Q$.
Therefore we obtain $\al\beta(\al - \beta) \xo = 0$, if we put $X$ as a unit vector field $\xo \in T_{\beta}$ into \eqref{eq: 3.4}.
 As we know $\alpha = -2 \tan(2r)$, $\beta=2 \cot (2r)$ where $r \in (0, \pi/4)$ on~$M_{B}$, we get a contradiction.
 This completes the proof of our Theorem~$1$.

\noindent Therefore we assert:
\begin{remark}\label{remark 1}
\rm The shape operator $A$ of a model space of Type~$(A)$ nor Type~$(B)$ in
$\GBt$ does not satisfy the semi-parallelism condition.
\end{remark}
\vskip 5pt
Summing up these discussions, we complete the proof of our Theorem~$1$ given in the introduction.\hspace{9.8cm}$\Box$

\vspace {0.15in}

\section{Semi-parallel structure Jacobi operator}\label{section 2}
\setcounter{equation}{0}
\renewcommand{\theequation}{2.\arabic{equation}}
\vspace{0.13in}

In this section, we give a complete prove our Theorem~$2$.
\noindent Suppose the structure Jacobi operator of $M$ has semi-parallelism, that is, $M$ satisfies the condition $(R(X,Y)\Rx) Z=0$. Besides, from the relation between \eqref{mcon} and \eqref{mcon 02} we see that the given condition is equivalent to
\begin{equation}\label{eq 4.1}
R(X,Y)(\Rx Z)=\Rx(R(X,Y)Z).
\end{equation}
The structure Jacobi operator $R_{\xi}$ is defined by $R_{\xi}(X) = R(X,{\xi}){\xi}$, where $R$ denotes the Riemannian curvature tensor on $M$. Then from the Gauss equation, it can be written as
\begin{equation}\label{eq: 4.1}
\begin{split}
R_{\xi} X  & = X - \eta(X) \xi + \eta(A\xi) AX - \eta(AX) A\xi\\
             & \quad - \sn \Big\{ \big(\EN (X) - \eta (X) \EN (\xi)\big) \xi_{\nu} + 3\EN
(\phi X) \PN \xi + \EN(\xi) \PN \phi X \Big\},\\
\end{split}
\end{equation}
where $\sn$ denotes from $\nu=1$ to $\nu=3$. From this, we see that $R_{\xi}\xi=0$.

\vskip 3pt

Put $Y=Z=\x$ into \eqref{eq 4.1}, due to $\Rx \x = 0$, we get:
\begin{equation}\label{eq: 4.2}
\Rx(\Rx X )=0.
\end{equation}
Using these observation from now on we show that $\x$ belongs to either $\Q$ or its
orthogonal complement $\QP$ such that $TM=\Q \oplus \QP$.
\begin{lemma}\label{lemma 4.1}
Let $M$ be a Hopf hypersurface in $\GBt$, $m \geq 3$, with \smsj. If the principal curvature $\alpha=g(A\x, \x)$ is constant along the direction of $\xi$, then $\xi$ belongs to either the distribution $\Q$ or the distribution~$\QP$.
\end{lemma}

\begin{proof}
Put $\xi$ satisfies \eqref{xi}
for some unit vectors $X_{0} \in \Q$ and $\xi_{1} \in \QP$.

\noindent Substituting $X=\xo$ in \eqref{eq: 4.1}, we have $\Rx({\xo})=\al^{2}\xo-\al^{2}\EKo\x$. This gives that
\begin{equation*}
\begin{split}
\Rx(\Rx \xo ) & = \Rx \big(\al^{2}\xo - \al^{2}\e(\xo)\x \big) \\
              & = \al^{2}\Rx \xo - \al^{2}\e(\xo)\Rx \x \\
              & = \al^{4}\xo - \al^{4}\EKo\x.
\end{split}
\end{equation*}
So, the condition of semi-parallel structure Jacobi operator implies
$$
\al^{4}\xo - \al^{4}\EKo\x=0.
$$
From this, taking the inner product with $\X \in \Q$, it gives $\al^{4}\EKo\e(\X)=0$. So we obtain the following three cases: $\al =0$, $\e(\X)=0$ or $\e(\xo)=0$.
When $\al$ is identically vanishing, by virtue of \eqref{eq: 1.4} we conclude that $\xi$ belongs to either $\Q$ or~$\QP$. For $\eta(\xi_{1})=0$,
then $\xi$ belongs to $\Q$ because of our notation~\eqref{xi}. Moreover, $\xi$ belongs to $\QP$ if $\E(X_0)=0$.
Accordingly, it completes the proof of our Lemma.
\end{proof}

According to Lemma~\ref{lemma 4.1}, we consider the case $\x \in \QP$.

\begin{lemma}\label{lemma 4.2}
Let $M$ be a Hopf hypersurface with \smsj \  in $\GBt$, $m \geq 3$. If the Reeb vector field $\xi$ belongs to the distribution~$\QP$, then $g(A\Q,\QP)=0$.
\end{lemma}

\begin{proof}
We may put $\xi=\xi_{1}$, because $\xi \in \QP$. Differentiating $\xi=\xi_{1}$ for any direction~$X$ on $M$, we obtain
\begin{equation}\label{differ}
\left\{
\begin{aligned}
& q_{2}(X)=2g(AX, \xt), \ q_{3}(X)=2g(AX, \xh)\ \text{and} \   \\
& AX=\e(AX)\x + 2g(AX, \xt)\xt + 2g(AX,\xh)\xh - \p \po AX \\
& \ \ \text{(or} \ AX= \e(X)A \x + 2\e_{2}(X) A \xt + 2 \e_{3}(X) A\xh - A \p \po X \text{)}.
\end{aligned}
\right.
\end{equation}

Putting $X=\xtw$ into \eqref{eq: 4.1}, it follows that $\Rx(\xtw)=2\xtw+\al A\xtw$. If the smooth function~$\al$ vanishes, it makes a contradiction. In fact, from \eqref{eq: 4.2} we see that $\Rx(\Rx \xtw)=4\xtw=0$. Thus we may consider that the smooth function $\al$ is non-vanishing.

\vskip 3pt

On the other hand, it follows that for any $W\in\Q$ the equation~\eqref{eq: 4.1} becomes
\begin{equation*}\label{eq: 4.4}
\Rx(W)=W+\po\p W+\al AW,
\end{equation*}
from this, together with the semi-parallelism of $R_{\xi}$, it follows that
\BE\label{eq: 4.5}
\begin{split}
0 & = \Rx(\Rx W)\\
  & = 2\al AW+2\al \eth(AW)\xth+2\al\etw(AW)\xtw-\al\po\p AW\\
    & \quad \ \ +\al^{2}A^{2}W+\al A\po\p W.
\end{split}
\EE
From \eqref{differ} and $\al \neq 0$, it follows that $2AW+\al A^{2}W=0$, where $AW=-A\po\p W$ for any tangent vector field $W \in \Q$.
 Taking the inner product with $\xtw$ and $\xth$, respectively, it becomes
\begin{equation}\label{eq: 4.7}
\al \etw(A^{2}W)  = -2\etw(AW), \quad   \al \eth(A^{2}W)  = -2\eth(AW).
\end{equation}

Moreover, according to \eqref{eq: 4.1}, we also have $\Rx(A\xtw)=2A\xtw+\al A^{2}\xtw$, which induces that
\begin{equation*}
\begin{split}
0=\Rx(\Rx \xtw )&=\Rx(2\xtw+\al A\xtw)\\
              &=2\Rx(\xtw)+\al\Rx( A\xtw)\\
              &=4\xtw+4\al A\xtw+\al^{2}A^{2}\xtw.
\end{split}
\end{equation*}
Again taking the inner product with $W\in\Q$ and using the fact $\al \neq 0$, we have
\begin{equation}\label{eq: 4.8}
\al \etw(A^{2}W)=-4\etw(AW).
\end{equation}
From this and \eqref{eq: 4.7}, we obtain $\etw(AW)=0$ for any tangent vector field $W \in \Q$.

Similarly, from \eqref{eq: 4.1} we get $\Rx\xth=2\xth+\al A\xth$ and $\Rx(A\xth)=2A\xth+\al A^{2}\xth$, which gives
\BE\label{eq: 4.10}
\begin{split}
0=\Rx(\Rx \xth )&=\Rx(2\xth+\al A\xth)\\
&=4\xth+4\al A\xth+\al^{2}A^{2}\xth.
\end{split}
\EE
From this, taking the inner product with $W\in\Q$ and using $\al \neq 0$, we have $4\eth(AX)+\al \eth(A^{2}X)=0$.
Combining this and \eqref{eq: 4.7}, we get also $\eth(AW)=0$ for any $W \in \Q$.
Until now, we have proven if $M$ satisfies our assumtpions, then the distribution $\QP$ is invariant under the shape operator,
that is, $g(A\Q,\QP)=0$. This gives a complete proof of our lemma.
\end{proof}

From this lemma and Theorem~$\rm A$ given by Berndt and Suh~\cite{BS1}, we see that a Hopf hypersurface $M$ satisfying the assumptions in Lemma~\ref{lemma 4.2}
is locally congruent to a model space of Type~$(A)$. Now, if  $\x$ belongs to $\Q$, then by virtue of Theorem~$\rm B$ a Hopf hypersurface $M$ with semi-parallel structure Jacobi operator is locally congruent to a real hypersurface of Type~$(B)$ in $\GBt$. Hence we conclude that \emph{let $M$ be a Hopf hypersurface in $\GBt$. If $M$ satisfies \eqref{eq 4.1} and $\x \al = 0$, then~$M$ is a model space of Type~$(A)$ or $(B)$.}

\vskip 3pt

From such a point of view, let us consider the converse problem. More precisely,
 we check whether the structure Jacobi operator $\Rx$ of a model space of Type~$(A)$ (or of Type~$(B)$, resp.) satisfies the semi-parallel condition~\eqref{eq 4.1}.

\vskip 3pt

In order to check our problem for a model space $M_{A}$,
we suppose that $M_{A}$ has the semi-parallel structure Jacobi operator. By virtue of Proposition~$\rm 3$ in \cite{BS1}, we see that $\x=\xo\in T_{\al}$ and $\x_{j} \in T_{\beta}$ for $j=2,3$. From this, the semi-parallel condition for $\Rx$ becomes
\begin{equation*}\label{eq: 4.11}
\begin{split}
\Rx(\Rx \xtw )&=4\xtw+4\al\beta\xtw+\al^{2}\beta^{2}\xtw \\
              &=(\al\beta+2)^{2}\xtw=0
\end{split}
\end{equation*}
when we put $X=\xtw$ in \eqref{eq: 4.2}. It implies $(\al\beta+2)=0$. But since $\al = \sqrt{8} \cot (\sqrt{8}r)$ and $\beta =\sqrt{2} \cot (\sqrt{2}r)$, we obtain $(\al\beta + 2) = 2\cot^{2}(\sqrt{2}r) \neq 0$ for $r \in (0, \pi / 2\sqrt{2})$. Thus it gives us a contradiction.

\vskip 3pt

In the sequel, we check whether $\Rx$ of a model space $M_{B}$ of Type~$(B)$ is semi-parallel.
To do this, we assume that $\Rx$ of $M_{B}$ satisfies the condition~\eqref{eq 4.1}. On a tangent vector space $T_{x}M_{B}$ at any point $x \in M_{B}$, the Reeb vector $\x$ belongs to $\Q$. From this and~\eqref{eq: 4.1}, the condition of \eqref{eq 4.1} implies that for $X=\xtw \in T_{\beta}$
\begin{equation*}
\Rx(\Rx \xtw )=\al^{2}\beta^{2}\xtw=0.
\end{equation*}
On the other hand, from \cite[Proposition~$\rm 2$]{BS1}, since $\alpha = -2 \tan(2r)$ and $\beta=2 \cot (2r)$ where $r \in (0, \pi / 4)$
on~$M_B$, we get $(\al\beta)^{2}=16$.
So, we consequently see that the tangent vector $\xtw$ must be zero, which gives a contradiction.

\vskip 5pt
\noindent Therefore we assert:
\begin{remark}\label{remark 2}
\rm The structure Jacobi operator $\Rx$ of a model space of Type~$(A)$ nor Type~$(B)$ in
$\GBt$ does not satisfy the semi-parallelism condition.
\end{remark}
Summing up these discussions, we complete the proof of our Theorem~$2$ given in the introduction.\hspace{9.8cm}$\Box$

\vspace{0.15in}

\section{Semi-parallel normal Jacobi operator}\label{section 3}
\setcounter{equation}{0}
\renewcommand{\theequation}{3.\arabic{equation}}
\vspace{0.13in}

Now, we observe a Hopf hypersurface $M$ in $\GBt$, $m \geq 3$, with semi-parallel normal Jacobi operator, that is, the normal Jacobi operator $\bar R_{N}$ of $M$ satisfies
\begin{equation*}
(R(X,Y)\bar R_{N})Z=0
\end{equation*}
for all tangent vector fields $X,Y,Z$ on $M$.

\vskip 3pt

In order to prove Theorem~3 mentioned in Introduction, let us consider the case that $M$ has vanishing geodesic Reeb flow.

\begin{lemma}\label{lemma 5.1}
Let $M$ be a real hypersurface in $\GBt$ with vanishing geodesic Reeb flow. If the normal Jacobi operator $\bar R_{N}$ of $M$ is semi-parallel, then $M$ is locally congruent to a model space of Type~$(A)$ or Type~$(B)$.
\end{lemma}

\begin{proof}
When the function ${\alpha}=g(A{\xi},{\xi})$ identically vanishes, it can be seen directly by \eqref{eq: 1.4} that $\xi$ can be divided into $\xi\in{\Q}$ or
$\xi\in\QP$. Then we first consider the case that $\xi$ belongs to $\Q$. By virtue of Theorem~$\rm B$, we get that $M$ is locally congruent to a model space of Type~$(B)$.

Next, we consider the case $\xi\in\QP$.
Substitution of the previous two relations in \cite[$(4.17)$]{PT} gives
\begin{equation}\label{eq: 3-1}
\begin{split}
&7W+7\al AW-6\po\p W \\
& \quad \quad =2\al \etw(AW)\xtw+2\al \eth(AW)\xth+\po\p(\po\p W)-\al \po\p AW.
\end{split}
\end{equation}
Since $\al=0$, it follows that $7W-6\po\p W=\po\p(\po\p W)$ for any  $W\in\Q$. Moreover, from ${\p}{\p}_{\nu}X ={\p}_{\nu}{\p}X+{\eta}_{\nu}(X){\xi}-{\eta}(X)
{\xi}_{\nu}$, $\nu=1,2,3$, we obtain $ {\p}{\p}_{1}({\p}{\p}_{1}W)
=W$. Thus \eqref{eq: 3-1} implies $\po\p W=W$. It implies $AW=0$ for any $W\in\Q$, together with~\eqref{differ}. It gives us a complete proof for $\al=0$.
\end{proof}

It remains to be checked if the normal Jacobi operator $\RN$ of a model space $M_{A}$ or $M_{B}$ satisfy the semi-parallelism condition.
For $\xi\in\QP$, we easily get $2\xi=0$ from \cite[Equations~$(5.2)$ and $(5.3)$]{PT}.
For $\xi\in\Q$, as we know $\alpha = -2 \tan(2r)$ with $r \in (0, \pi/4)$ on a real hypersurface of Type~$(B)$, $\al$ never vanishes (see \cite[Proposition~$2$]{BS1}).
So, neither the normal Jacobi operator $\RN$ of $M_{A}$ nor $M_{B}$ does not satisfy the semi-parallelism condition.
 Thus we get the following:

\begin{corollary}\label{remark 3}
Let $M$ be a real hypersurface in $\GBt$, $m \geq 3$, with vanishing geodesic Reeb flow. Then there does not exist any Hopf hypersurface if the normal Jacobi operator $\bar R_{N}$ of $M$ satisfies the condition of semi-parallelism.
\end{corollary}
\vskip 5pt
Combining Theorem~$\rm C$ and Corollary~\ref{remark 3}, we give a complete proof of Theorem~$3$ in the introduction. \hspace{9.3cm}$\Box$

\vskip 15pt

\noindent {\bf Acknowledgements.}\quad The authors would like to express their deep gratitude to Professors Y.J. Suh and J.D. P\'{e}rez for their suggestions to solve this problem and nice comments with their best effort.

\vspace{0.15in}


\end{document}